\documentclass[11pt,reqno]{amsart}
\usepackage{amsmath}
\usepackage{amssymb,amscd}
\usepackage{graphicx}

\numberwithin{equation}{section}

\newtheorem{theorem}{Theorem}[section]

\newtheorem{lemma}[theorem]{Lemma}

\theoremstyle{definition}
\newtheorem{defn}[equation]{Definition}
\newtheorem{remark}[equation]{Remark}

\def \bb{\mathbb}
\def \mb{\mathbf}
\def \mc{\mathcal}
\def \mf{\mathfrak}

\def \ZZ{{\bb{Z}}}

\def \({\left(}
\def \){\right)}
\def \<{\langle}
\def \>{\rangle}

\def \tensor{\otimes}

\begin{document}

\title[Tannakian classification of toric principal bundles]{Tannakian classification of 
equivariant principal bundles on toric varieties}

\author[I. Biswas]{Indranil Biswas}
\address{School of Mathematics, Tata Institute of Fundamental Research, Mumbai, India }

\email{indranil@math.tifr.res.in}

\author[A. Dey]{Arijit Dey}

\address{Department of Mathematics, Indian Institute of Technology-Madras, Chennai, India }

\email{arijitdey@gmail.com}

\author[M. Poddar]{Mainak Poddar}

\address{Mathematics Group, Middle East Technical University, Northern Cyprus Campus, Guzelyurt, Mersin 10, Turkey}

\curraddr{Department of Mathematics, Indian Institute of Science Education and Research (IISER) Pune, India }

\email{mainakp@gmail.com}

\subjclass[2010]{Primary: 14M25, 32L05. Secondary: 14L30}

\keywords{Equivariant bundles, principal bundles, toric varieties, Tannakian category.}

\begin{abstract}
Let $X$ be a complete toric variety equipped with the action of a torus $T$
and $G$ a reductive algebraic group, defined over an algebraically closed field $K$. 
We introduce the notion of a compatible $\Sigma$--filtered algebra associated to $X$, 
generalizing the notion of a compatible $\Sigma$--filtered vector space due to Klyachko, 
where $\Sigma$ denotes the fan of $X$.
We combine Klyachko's classification of $T$--equivariant vector bundles on $X$ with Nori's
Tannakian approach to principal $G$--bundles, to give an equivalence of categories between
 $T$--equivariant principal $G$--bundles on $X$ and certain compatible $\Sigma$--filtered algebras 
associated to $X$, when the characteristic of $K$ is $0$. 
 
\end{abstract}

\maketitle

\tableofcontents

\section{Introduction}

Let $X$ be a toric variety, with fan $\Sigma$, under the action of a torus $T$, and let $G$ be a reductive 
algebraic group; all are defined over an algebraically closed field $K$. A
$T$--equivariant vector bundle $E$ on $X$ is a vector bundle on $X$ endowed with a lift of the $T$--action which is linear on fibers. 
The $T$--equivariant vector bundles over a nonsingular toric variety were 
first classified by Kaneyama \cite{Kan1}. This classification result for toric vector 
bundles is up to isomorphism and it involves both combinatorial and linear algebraic 
data modulo an equivalence relation. Recently this work has been generalized for 
$T$--equivariant principal $G$--bundles \cite{BDP1, BDP2}; also see 
\cite{BDP3, DP}, when $K$ is the field $\mathbb{C}$ of complex numbers.

In a foundational paper Klyachko gave an alternative description of equivariant vector bundles on arbitrary 
toric varieties (possibly non-smooth) defined over any algebraically closed field \cite{Kly}. His 
correspondence gives an equivalence between the category $\mf{Vec}^T (X)$ of equivariant 
vector bundles on $X$ and the category $\mf{Cvec}(\Sigma)$ of finite dimensional vector 
spaces with collection of decreasing $\mathbb{Z}$--graded filtrations, indexed by the
rays of $\Sigma$, satisfying a certain compatibility condition. Klyachko 
used his classification theorem to compute the Chern characters and sheaf cohomology of 
equivariant vector bundles. As a major application, later he used his classification 
of equivariant vector bundles over $\mathbb P^2$ to prove Horn's conjecture 
on eigenvalues of sums of Hermitian matrices \cite{Kly2}. Another interesting and more recent application 
is a theorem of Payne \cite{Pay} that the moduli space of rank 3 toric vector bundles satisfy Murphy's law. Klyachko's classification theorem has also been generalized for equivariant torsion-free and equivariant pure sheaves by Perling \cite{Per} and Kool \cite{Ko} respectively.

In this paper, our aim is to prove an analogue of Klyachko's result for $T$--equivariant principal 
$G$--bundles over a toric variety $X$. For some technical reasons (see Lemma \ref{cext2}), we need the group 
$G$ to be linearly reductive. Hence, for the main result (Theorem \ref{classi}), we find it natural to 
restrict ourselves to the case where the characteristic of $K$ is zero. However, the result may be generalized to some fields of positive characteristic; see \cite[Chapter 6, Theorem 2]{Nag}.
 
The first step in our formulation is 
an equivariant Nori theorem (Theorem \ref{equiv1}), where we identify $T$--equivariant principal $G$--bundles
with functors from the category of finite dimensional $G$--modules to the category
of $T$--equivariant vector bundles over $X$, satisfying Nori's four
conditions; see Section \ref{NC}. In fact this theorem holds not only for $T$, but for any affine algebraic group $\Gamma$ acting on an algebraic variety $X$.

Next we introduce the notion of a compatible $\Sigma$--filtered $K$--algebra
which is a $K$--algebra endowed with a 
collection of decreasing $\mathbb{Z}$--graded filtrations indexed by the rays of $\Sigma$, 
that satisfy certain additive and multiplicative compatibility conditions; see Definitions \ref{sfa} and \ref{csfvs}.
Let $\mf{Calg}_G(\Sigma)$ be the category of such filtered $K$--algebras,
$G$--equivariantly isomorphic to $K[G]$ (under the standard action),
that satisfy the following: For every top dimensional cone $\sigma\,\in\,\Sigma$, the $K$--algebra 
admits an action of $T$ which is compatible with the filtrations and commutes with 
the $G$--action. 

Then we invoke a crucial fact, that a $T$--equivariant principal $G$--bundle over any affine toric variety is equivariantly trivializable. In the complex case, a proof of this is given in \cite[Theorem 2.1]{BDP2} using an equivariant Oka-Grauert principle \cite{HK}. See \cite{KLS1, KLS2, KLS3, KLS4} for  stronger versions of equivariant Oka-Grauert principle and related results. 
The proof  presented here  (Theorem \ref{etriv}), which works for arbitrary characteristic, is due to one of the referees. It is based on a characteristic-free version of Luna's \'etale slice theorem
 proved in \cite{BR}.

Now, assume that every maximal cone in the fan $\Sigma$ of $X$ is of top dimension.
This holds, in particular, when $X$ is complete.
Using the two results mentioned above, we prove an 
equivalence between the category $\mf{Pbun}^T_G(X)$ of $T$--equivariant principal $G$--bundles
over $X$, and the category $\mf{Calg}_G(\Sigma)$ (Theorem \ref{classi}).
The most intriguing step in our proof is the commutativity of the $T$ and $G$ actions on the 
$K$--algebras in the definition of $\mf{Calg}_G(\Sigma)$; see Lemma \ref{commute}.
As a corollary to Theorem \ref{classi}, we obtain a necessary and sufficient condition for
an equivariant reduction of structure group (Theorem \ref{reduc}).
When $G\,=\,{\rm GL}(n, K)$, Klyachko's filtration data for equivariant vector 
bundle may be recovered from our filtered algebra description (see the proof of Lemma \ref{surj}). 

 In a recent work \cite{IS}, Ilten and S$\ddot{\text{u}}$ss have obtained a Klyachko--type 
classification of torus equivariant vector bundles over $T$--varieties, and related 
it to Hartshorne's conjecture on splitting of rank two bundles over projective 
spaces. It seems natural that our classification 
of equivariant principal $G$-bundles should generalize for $T$-varieties. 

Recently Kaveh and Manon shared with us another interesting
approach to toric principal $G$--bundles \cite{KM}.

\section{Nori's correspondence}\label{NC}

Let $X$ be a separated, integral, finite type scheme over an arbitrary field $K$. We denote the ring of $K$-valued regular functions on $X$ by $K[X]$. 

The category of finite dimensional vector spaces over $K$ will be denoted by $\mf{Vec}$.
Let $G$ be an affine algebraic group over $K$. By an algebraic group we mean a smooth finitely
generated group scheme over $K$; see \cite[Definition 0.2, pp. 2]{Mum}.

The category of algebraic left representations of $G$ that are finite dimensional $K$--vector
spaces will be denoted by $G$--mod.

For convenience of readers, we recall the following equivalent notions of categorical equivalence which will be useful later (see \cite[Chapter IV, section 4]{Mc} for details).
\begin{enumerate}
\item{} Two categories $\mathcal C_1$ and $\mathcal C_2$ are equivalent, i.e. there exists functors $F_1\,: 
\mathcal C_1 \longrightarrow \mathcal C_2$ and $F_2\,: \mathcal C_2 \longrightarrow \mathcal C_1$ such that $F_2 \circ 
F_1$ and $F_1\circ F_2$ are naturally isomorphic to identity functors $\text{Id}_{\mathcal C_1}$ and 
$\text{Id}_{\mathcal C_2}$ respectively. $F_1$ and $F_2$ are called quasi-inverses of each other.

\item{} There exists a functor $F\,: \mathcal C_1 \longrightarrow \mathcal C_2$, which is full, faithful and essentially surjective. 
\end{enumerate}

Let $\mf{T}$ be a tensor category 
over $K$ (see \cite[Chapter II, section 1]{LNM900} for definition of tensor category). A functor $$\mb{H}\,:\, G\text{--mod}\,\longrightarrow\, \mf{T}$$ is said to satisfy
properties F1--F4 if the following hold (see \cite[Chapter 1]{Nor2} for more detailed description of
these properties):
\begin{enumerate}
 \item F1: $\mb{H}$ is a $K$--additive exact functor,
 \item F2: $\mb{H} \circ \otimes \,=\, \otimes \circ(\mb{H} \times
 \mb{H})$,
 \item F3: furthermore,
 \begin{enumerate} \item $\mb{H}$ respects associativity
 of tensor products,
 \item $\mb{H}$ respects commutativity
 of tensor products, 
 \item $\mb{H}$ takes the identity object of ($G$--mod,\, $\otimes$), namely the trivial
$G$--module $K$, to the identity object of ($\mf{T},\, \otimes$), and
\end{enumerate}
\item F4: the functor $\mb{H}$ is faithful.
 \end{enumerate}
 
Let $\mf{Vec}(X)$
be the category of vector bundles over $X$. 
Whenever convenient, we shall identify a vector bundle on $X$ with the
locally free coherent sheaf on $X$ given by its local sections.
Consider the category $\mathfrak{Nor}(X)$ of ``Nori functors'' whose
\begin{itemize}
\item objects are functors $\mb{E}\,:\, G\text{--mod} \,\longrightarrow\, \mf{Vec}(X)$
that satisfy F1-F4, and

\item morphisms are natural isomorphisms of functors.
\end{itemize}

Let $\mathfrak{Pbun}_G(X)$ denote the category of principal $G$--bundles
over $X$. Given a principal $G$--bundle $E_G$, 
the functor which sends a $G$-module $V$ to the associated vector bundle $E_G \times^G V$
satisfies properties F1-F4. We may therefore refer to it as the Nori functor associated to $E_G$.
Let
\begin{equation}\label{a2}
\mb{N}_0\,:\, \mathfrak{Pbun}_G(X) \,\longrightarrow \mathfrak{Nor}(X)
\end{equation}
be the functor that sends any principal $G$--bundle $E_G$ to its associated Nori functor.

Let $\mf{Qco}(X)$ be the category of quasi-coherent sheaves of $\mc{O}_X$--modules.
In \cite[Lemma 2.2]{Nor2}, Nori showed that any functor $\mb{E}\,\in\, \mathfrak{Nor}(X)$ admits
a unique and natural extension to a functor $\mb{\overline{E}}$ from affine
$G$--schemes to $\mf{Qco}(X)$. He showed that $\mb{\overline{E}}(G)$ is a
principal $G$--bundle over $X$. This defines a functor 
\begin{equation}\label{a3}
\mb{N}_1\,:\, \mathfrak{Nor}(X) \,\longrightarrow\, \mathfrak{Pbun}_G(X)
\end{equation}
that sends a Nori functor
$\mb{E}$ to the principal $G$--bundle $\mb{\overline{E}}(G)$.
He went on to show that $\mb{N}_0$ and $\mb{N}_1$ are quasi-inverses,
proving that the categories $\mathfrak{Nor}(X)$ and $\mathfrak{Pbun}_G(X)$ are equivalent.

In this section, we shall establish an equivariant analogue of the above equivalence.
Let $\Gamma$ be an affine algebraic group defined over $K$, and let
$$
\eta\,:\, \Gamma\times X\,\longrightarrow\, X
$$
be an algebraic left action of $\Gamma$ on $X$ (see \cite[Definition 0.3, pp. 2]{Mum} for definition of group 
scheme action on an arbitrary scheme). A $\Gamma$--{\it
equivariant} vector bundle on $X$ is a pair $(W\,
,\widetilde{\eta})$, where $W$ is an algebraic vector bundle on
$X$ and
$$
\widetilde{\eta}\,:\, \Gamma\times W\,\longrightarrow\, W
$$
is an algebraic left action of $\Gamma$ on the total space of $W$
such that
\begin{itemize}
\item $\widetilde{\eta}$ is a lift of $\eta$, and

\item $\widetilde{\eta}$ preserves the linear structure on $W$, in particular,
it is fiberwise linear.
\end{itemize}
We refer the reader to \cite[section 1.2]{Thom} for a detailed definition of an equivariant sheaf. 

Similarly, a $\Gamma$--{\it equivariant} principal $G$--bundle on $X$
is a pair $(E_G\, ,\widetilde{\eta})$, where $E_G$ is an algebraic
principal bundle on $X$ and
$$
\widetilde{\eta}\,:\, \Gamma\times E_G\,\longrightarrow\, E_G
$$
is an algebraic left action of $\Gamma$ on the total space of
$E_G$ such that
\begin{itemize}
\item $\widetilde{\eta}$ is a lift of $\eta$, and

\item $\widetilde{\eta}$ commutes with the right action of $G$ on $E_G$.
\end{itemize}

Let $\mf{Vec}^{\Gamma}(X)$ (respectively, $\mf{Pbun}_G^{\Gamma}(X)$) be the
category of $\Gamma$--equivariant vector bundles (respectively,
$\Gamma$--equivariant principal $G$--bundles) over $X$.
Let
\begin{equation}\label{not2}
\mf{Nor}^{\Gamma}(X)
\end{equation}
be the category whose
\begin{itemize}

\item objects are functors $\mb{E}\,:\, G\text{--mod} \,\longrightarrow\,
\mf{Vec}^{\Gamma}(X)$ satisfying F1--F4, and

\item morphisms are natural isomorphisms of functors.
\end{itemize}

Take any $\mb{E}\, \in\, \mf{Nor}^{\Gamma}(X)$.
For any $V \,\in\, G\text{--mod}$, let $E(V)$ denote the underlying vector bundle of $\mb{E}(V)$,
and let $\widetilde{\eta}(V)$ denote the action
of $\Gamma$ on $E(V)$. For any homomorphism
of $G$--modules $\phi\,:\, V \,\longrightarrow\, W$, the following diagram is commutative:
\begin{equation}\label{cdf}
\begin{CD}
\Gamma \times E(V) @> \widetilde{\eta}(V) >> E(V) \\
@V {\rm id} \times {\mb E}(\phi) VV @V {\mb E}(\phi) VV \\
 \Gamma \times E(W) @> \widetilde{\eta}(W) >> {E}(W)
\end{CD}
\end{equation}
Also, we have $E(V \otimes W) \,= \,E(V) \otimes E(W)$ and
$\widetilde{\eta}(V \otimes W) \,=\, \widetilde{\eta}(V) \otimes
\widetilde{\eta}(W)$.

We say that a (possibly infinite dimensional)
$G$--module $\overline{V}$ is \textit{locally finite} if
given any vector $v \,\in\, \overline{V}$, there exists a finite dimensional
$G$--submodule $V \,\subset\, \overline{V}$ with $v\,\in\, V$. Let
$G\text{--}\overline{\text{mod}}$ denote the category whose objects are
locally finite $G$--modules and morphisms are $G$--module
homomorphisms.

It is well-known that for any affine algebraic group $G$ and any affine $G$--scheme $X$, the
$G$--module $K[X]$ is locally finite (cf. \cite[Proposition 8.6]{Hum}, 
\cite[Theorem, section 3.3]{Wa}).

\begin{lemma}\label{cext}
Let ${\mf{Alg}}^{\Gamma}(X)$ denote the category
of $\Gamma$--equivariant sheaves of commutative associative $\mc{O}_X$--algebras, and let
$G-{\rm sch} $ be the category of affine $G$--schemes. Let $\mb{E} $ be an object of
$\mf{Nor}^{\Gamma}(X)$. Then there exists a unique extension of
$\mb{E}$ to a functor $$\mb{\overline{E}}\,:\, G {\rm -sch} \,\longrightarrow\,
{\mf{Alg}}^{\Gamma}(X) \,.$$
\end{lemma}

\begin{proof} Let $\mf{Qco}^{\Gamma}(X)$ be the category of all $\Gamma$--equivariant
quasicoherent sheaves of $\mc{O}_X$--modules. First, observe that there is a unique
 extension $\mb{\overline{E}} \,:\, G\text{--}\overline{\rm mod} \,
\longrightarrow\,\mf{Qco}^{\Gamma}(X)$ that satisfies properties F1-F4 (cf. \cite[Lemma(2.1)]{Nor2}).
For $\overline{V}\,\in\, G\text{--}\overline{\rm mod}$, denote the underlying sheaf 
of $\mb{\overline{E}}(\overline{V})$ by ${\overline{E}}(\overline{V})$. Note that ${\overline{E}}(\overline{V})$
is the direct limit of ${E}(V)$, where $V$ varies over all finite dimensional
$G$--submodules of $\overline{V}$.

Use \eqref{cdf} to
 take direct limit
of the morphisms $\widetilde\eta(V) \,:\, \Gamma \times E(V)\,\longrightarrow\, E(V)$
as $V$ varies over all finite dimensional $G$--submodules of $\overline{V}$.
In this way we obtain an action
\begin{equation}\label{a1}
\widetilde{\eta}(\overline{V})\,:\, \Gamma
\times \overline{E}(\overline{V}) \,\longrightarrow\, \overline{E}(\overline{V})\, .
\end{equation}

Suppose $\overline{\phi}\,: \,\overline{U} \,\longrightarrow\, \overline{V}$ is a morphism
of locally finite $G$--modules. To define $\mb{\overline{E}}(\overline{\phi})$,
consider any $u \,\in \,\overline{U}$. There exists a finite
dimensional $G$--module $U \,\subset\, \overline{U}$ such that $u \,\in\, U$. 
Let $V$ denote the image $\overline{\phi}(U)$ with $i_V \, :\, V \,\longrightarrow\,
\overline{V}$ being the inclusion map. Note that $V$ is a finite dimensional
$G$--module. Let $\psi\,:\,U \,\longrightarrow\,V$ be the unique homomorphism such that
$\overline{\phi}\vert_{U}\,=\, i_V \circ \psi$. Define
$$\mb{E}(\overline{\phi}) (u) \,=\, [\mb{E}(\psi)(u)] \, ,$$ to be the equivalence
class of $\mb{E}(\psi)(u) \,\in\, V$
in the direct limit $\overline{V}$. It is straightforward to check that this is indeed well-defined.
Since the operation of direct limit commutes with tensor product, the extension
preserves tensor product. 

Following Nori, consider a commutative $G$--algebra $A$
as a locally finite $G$--module together with a homomorphism $m \,:\, A \otimes A
\,\longrightarrow\, A$. Then $\mb{\overline{E}}(m)$ defines the structure of a $\Gamma$--equivariant
commutative, associative $\mc{O}_X$--algebra on $\overline{E}(A)$.

A similar argument shows that if $\phi\,:\, A \,\longrightarrow\, B$ is a homomorphism
of $G$--algebras, then $\mb{\overline{E}}(\phi)$ is a homomorphism of
$\Gamma$--equivariant sheaves of $\mc{O}_X$--algebras.
\end{proof}

It was shown by Nori, \cite[Lemma 2.3]{Nor2}, that $\mb{\overline{E}}(K[G])$ is a sheaf of $\mc{O}_X$--algebras that 
corresponds to a principal $G$--bundle over $X$. We denote by
${\overline{E}}(K[G])$ the principal $G$--bundle on $X$ corresponding
to $\mb{\overline{E}}(K[G])$.

The right $G$--action on ${\overline{E}}(K[G])$ is constructed as follows. Consider $G'$
to be a copy of $G$ with trivial $G$--action. Note that
${\overline{E}}(K[G'])$ is the trivial principal $G'$--bundle $X \times G'\,\longrightarrow\, X$ with trivial
$\Gamma$--action on fibers. Let
\begin{equation}\label{trivac1}
a\,:\, G \times G' \,\longrightarrow\, G
\end{equation}
be the multiplication map of $G$. This $a$ produces an action of $G'$
on $G$. Then $\mb{\overline{E}}(a)$ induces a morphism
\begin{equation}\label{trivac2}
{\overline{E}}(a) \,:\, {\overline{E}}(K[G]) \times_X{\overline{E}}(K[G'])\,=\,
{\overline{E}}(K[G]) \times G' \,\longrightarrow\, {\overline{E}}(K[G])\,.
\end{equation}

This induces the required fiber-wise action of $G'$ on
${\overline{E}}(K[G])$. Note that $ \mb{\overline{E}}(a)$ is a morphism of
$\Gamma$--equivariant sheaves. Therefore, the actions of
$\Gamma$ and $G'$ on $ {\overline{E}}(K[G])$ commute. Consequently, we have $ {\overline{E}}(K[G])\,\in\, \mf{Pbun}_G^{\Gamma}(X)$.

It follows that $\mb{N_1}$ in \eqref{a3} produces a functor
$$ \mb{N}_1^{\Gamma} \,:\, \mf{Nor}^{\Gamma}(X) \,\longrightarrow\, \mf{Pbun}_G^{\Gamma}(X)\, ,\,~~
\mb{E} \,\longmapsto\, (\overline{E}(K[G])\, ,
\widetilde{\eta}(K[G]))\, ,$$ where $\widetilde{\eta}$ is constructed in \eqref{a1}.
On the other hand,
the functor $\mb{N}_0$ in \eqref{a2} produces a functor
$\mb{N}_0^{\Gamma} \,:\, \mf{Pbun}_G^{\Gamma}(X) \,\longrightarrow\, \mf{Nor}^{\Gamma}( X)$.

An analogue of the following result when $\Gamma$ is a finite group has appeared before in \cite{BBN}.

\begin{theorem}\label{equiv1} The above two functors $\mb{N}_0^{\Gamma}$ and
$\mb{N}_1^{\Gamma}$ are mutually
quasi-inverses that induce an equivalence of categories between
$\mf{Pbun}_G^{\Gamma}(X)$ and $\mf{Nor}^{\Gamma}(X)$.
\end{theorem}

\begin{proof} The proof follows verbatim from \cite[Proposition 2.5 and Lemmas 2.6, 2.7, 2.8]{Nor2} once one replaces the category of sheaves of quasi-coherent $\mathcal{O}_X$-modules with its $\Gamma$-equivariant analogue.
\end{proof}

\section{Filtration functor for vector bundles}

Let $X$ be a toric variety defined over $K$, corresponding to a fan $\Sigma$ in a lattice $N$ (see \cite{Dan, Ful, Oda} for details). 
Let $T$ denote the algebraic torus whose one-parameter subgroups are indexed by $N$.
Then $X$ admits an action of $T$ with an open dense $T$--orbit $O$. Let $n\,=\,\text{dim}(X)\,=\,\text{dim}(T)$. 
Denote the set of all $d$--dimensional cones of $\Sigma$ by
$\Sigma(d)$. Let $|\Sigma(1)|$ be the set of primitive integral generators 
of elements of $\Sigma(1)$.

Define $M \,=\, \text{Hom}_{\ZZ}(N, \ZZ)$. 
Then $M$ is isomorphic to the group of characters of $T$.
For any
$\sigma \,\in\, \Sigma$, denote the corresponding affine toric subvariety of $X$ by $X_{\sigma}$; also 
define $$\sigma^{\perp}\,= \,\{u \,\in\, M
\,\mid\, u(n)\,=\, 0 \, \ \forall \ n \,\in\, \sigma\}$$ and
$M_{\sigma} \,:= \,M/{\sigma^{\perp}}$. Then $M_{\sigma}$ is the
character group of the maximal sub-torus $T_{\sigma} \,\subset\, T$ that has a fixed point in 
$X_{\sigma}$. 
Let $\overline{\mf{Vec}}$ be the category of $K$--vector spaces of
countable dimension; the morphisms are $K$--linear homomorphisms.

\begin{defn} A {\it decreasing filtration} $\mc{V}$ on a $K$--vector space $V
\,\in\, \mf{\overline{Vec}}$ is a
collection $\{ V(i)\,\mid\, i \,\in\, \mathbb{Z}\}$ of subspaces of $V$ such that
$V(i) \,\supseteq\, V(i+1)$ for each $i$.
We say $\mc{V}$ is full if given any $v\,\in\, V$ there exists an integer $i$
depending on $v$ such that $v \,\in\, V(i)$.
\end{defn}

\begin{defn}\label{sfvs}
A $\Sigma$--{\it filtration} on a vector space $V\,\in\, \mf{\overline{Vec}}$ is a collection of
 full decreasing filtrations
$$ \mc{V}^{\rho} \,:\ ~ \cdots \,\supseteq\,V^{\rho}(i-1) \,\supseteq\, V^{\rho}(i)\,
\supseteq\, V^{\rho}(i+1) \,\supseteq\,\cdots\, , $$
on $V$, where $\rho \,\in\, |\Sigma(1)|$. We denote the data $(V, \{V^{\rho}(i) \} )$ by $\mc{V}^{\bullet}$
and say that $\mc{V}^{\bullet}$ is a $\Sigma$--filtered vector space
on $X$.
If the vector space $V$ is finite dimensional then $\mc{V}^{\bullet}$ is said to be
finite dimensional.

A {\it morphism} of $\Sigma$--filtered vector spaces $\phi\,:\, \mc{V}^{\bullet}
\,\longrightarrow\, \mc{W}^{\bullet}$ is a homomorphism of vector
spaces $\phi\,:\, V \,\longrightarrow\, W$, such that
$\phi(V^{\rho}(i) )\,\subseteq\, W^{\rho}(i) $ for each $i$ and $\rho$. Such a morphism
is injective (respectively, surjective) if the underlying
homomorphism of vector spaces is injective (respectively, surjective).
\end{defn}

The category of $\Sigma$--filtered vector spaces is a tensor category with the following
 tensor product: 
\begin{equation}\label{ten0}
\mc{V}^{\bullet} \otimes \mc{W}^{\bullet} \,=\, \{ V\otimes W, (V\otimes W)_{\dagger}^{\rho}(j) \} 
\end{equation}
 where
\begin{equation}\label{ten}
(V\otimes W)_{\dagger}^{\rho}(j) \,=\, \sum_{p+q=j} V^{\rho}(p)
\otimes W^{\rho}(q)\,. \end{equation}

\begin{defn}\label{sfa}
Let $(A\, ,m)$ be a commutative, associative, $K$--algebra, where $A \,\in\, \mf{\overline{Vec}}$ is the
underlying $K$--vector space of the algebra, and $m\,:\, A \tensor A\,\longrightarrow\, A$ is the
multiplication operation. A $\Sigma$--filtration on $(A,\,m)$ is a
$\Sigma$--filtration $\mc{A}^{\bullet}\,= \,(A,\, \{ A^{\rho}(i) \})$ on the vector space $A$
such that
$$m(A^{\rho}(i) \tensor A^{\rho}(j) ) \,\subseteq\, A^{\rho}(i+j) $$
for every $\rho \,\in \, |\Sigma(1)|$ and $i\, ,j \,\in\, \mathbb{Z}$.

The above data $ (A,\, \{A^{\rho}(i) \},\, m) $ is denoted by $(\mc{A}^{\bullet},\, m)$, and is
called a $\Sigma$--filtered algebra on $X$.

A morphism of $\Sigma$--filtered algebras $(\mc{A}^{\bullet}_1, \, 
m_1) \,\longrightarrow\, (\mc{A}^{\bullet}_2,\, m_2)$ is a homomorphism of underlying
$K$--algebras that respects the filtrations. Equivalently, it is a
morphism of $\Sigma$--filtered vector spaces $$\phi\,:\,
\mc{A}^{\bullet}_1 \,\longrightarrow\, \mc{A}^{\bullet}_2 $$ such that $\phi \circ
m_1 \,=\, m_2 \circ (\phi \otimes \phi)$.
\end{defn}

\begin{defn}\label{csfvs}
A {\it compatible} $\Sigma$--filtered vector space on $X$ is a $\Sigma$--filtered vector
space $\mc{F}^{\bullet}\,=\, (F, \,\{F^{\rho}(i) \})$ such that
for every $\sigma \,\in\, \Sigma$, there exists a decomposition of the vector space
\begin{equation}\label{ca1}
F \,=\, \bigoplus_{[u] \in M_{\sigma}} F^{\sigma}_{[u]}\, ,
\end{equation}
with the following property: For each $\sigma$ and for each $\rho\,\in\,
\sigma \bigcap |\Sigma(1)|$
\begin{equation}\label{ca2}
F^{\rho} (i)\,= \,\bigoplus_{u(\rho) \ge i} F^{\sigma}_{[u]}\, .
\end{equation}

Similarly a compatible $\Sigma$--filtered algebra on $X$ is a $\Sigma$--filtered
algebra $(\mc{F}^{\bullet},\, m) $ whose underlying
$\Sigma$--filtered vector space $\mc{F}^{\bullet}$ is compatible, and the subspaces $F^{\sigma}_{[u]}$ in \eqref{ca1} satisfy
\begin{equation}\label{ca3}
\sum_{[u] + [v] \,=\, [w]} m ( F^{\sigma}_{[u]}
\otimes F^{\sigma}_{[v]} )\,\subseteq\,F^{\sigma}_{[w]}\, .
\end{equation}

A {\it morphism} between compatible $\Sigma$--filtered vector spaces (respectively, algebras)
is simply a morphism between the
underlying $\Sigma$--filtered vector spaces (respectively, algebras).
\end{defn}

\begin{remark}\label{remcomp} In the above definition it is enough to require that 
a decomposition \eqref{ca1} satisfying \eqref{ca2} exists for every maximal cone 
$\sigma$ in the fan $\Sigma$. A decomposition corresponding to a maximal cone
induces decompositions corresponding to its subcones.
\end{remark}

\begin{remark}\label{remact} 
A decomposition as in \eqref{ca1} corresponds to an action of $T_{\sigma} $ on $F$.
\end{remark}

Given a $\Sigma$--filtered vector space $(F,\, \{F^{\rho}(i) \})$, a decomposition \eqref{ca1} that satisfies 
\eqref{ca2}, will be called a {\it compatible decomposition}.

Let $\mf{\overline{Fvec}}(\Sigma)$ and $\mf{\overline{Cvec}}(\Sigma)$ denote the categories of 
$\Sigma$--filtered vector spaces and compatible $\Sigma$--filtered vector spaces on 
$X$ respectively. Their finite dimensional counterparts are denoted by
\begin{equation}\label{not1}
\mf{Fvec}(\Sigma)\ \ \text{ and } \ \ \mf{Cvec}(\Sigma)
\end{equation}
respectively. These are additive categories (see \cite[Section12.3]{stacks} for a definition).

The category $\mf{\overline{Cvec}}(\Sigma)$ is a tensor category with product
as in \eqref{ten}: Suppose $\mc{V}^{\bullet}$ and
$\mc{W}^{\bullet}$ are compatible $\Sigma$--filtered vector spaces.
Let $V \,= \bigoplus_{[u] \in M_{\sigma}} V^{\sigma}_{[u]} $ and
$W \,= \,\bigoplus_{[u] \in M_{\sigma}} W^{\sigma}_{[u]}$ be
compatible decompositions for $\mc{V}^{\bullet}$ and
$\mc{W}^{\bullet}$ respectively. Define
$$ (V \otimes W)^{\sigma}_{[u]} \,=\,
\bigoplus_{[u_1]+[u_2]=[u]} V^{\sigma}_{[u_1]} \otimes
W^{\sigma}_{[u_2]} \, .$$ Then 
$$V
\otimes W \,=\, \bigoplus_{[u] \in M_{\sigma} } (V \otimes
W)^{\sigma}_{[u]} $$ is a compatible decomposition for
$\mc{V}^{\bullet} \otimes \mc{W}^{\bullet}$.

Let $\mf{Vec}^T(X)$ (respectively, $\mf{Pbun}_G^{T}(X)$) denote the category of $T$--equivariant vector
bundles (respectively, $T$--equivariant principal $G$--bundles) on $X$. There exists a fully
faithful, surjective functor
$$\mb{F}\,:\, \mf{Vec}^T(X)\,\longrightarrow \,\mf{Cvec}(\Sigma)$$
(see \cite[Theorem 2.2.1]{Kly}). 
We shall sketch the construction of $\mb{F}$.
Let $\xi \in \mf{Vec}^T(X) $ be a bundle of rank $r$.
Fix a closed point $x_0$ in the open $T$--orbit $O\,\subset\, X$. Denote by $F$ the
fiber $\xi(x_0)$.
Let $\sigma$ be a cone of $\Sigma$ and $X_{\sigma}$ the corresponding affine toric variety.
Denote by $\xi_{\sigma}$ the restriction of $\xi $ to $X_{\sigma}$.
Consider the action of $T$ on the space of sections of
$\xi_{\sigma}$ defined by
$$
(t \cdot s)(x) \,=\, t s(t^{-1} x)
$$
for any point $x\in X_{\sigma}$, any element $t \in T$, and any section
$s$ of $\xi_{\sigma}$. A section $s$ is said to be semi-invariant if $t\cdot s \,=\, u(t) s$
for some character $u$ of $T$.

 It was shown by Klyachko, \cite[Proposition 2.1.1]{Kly}, that there exists
 a framing (which is not unique) of $\xi_{\sigma}$ by semi-invariant sections. 
 Fix such a framing $(s_1, \ldots, s_r)$.
 Let
$S_{\sigma}$ be the $T$--submodule of
 $H^{0}(X_{\sigma},\, \xi_{\sigma})$ generated by the semi-invariant sections $s_1, \ldots, s_r$.
Evaluation at $x_0$
gives an isomorphism of vector spaces $ev_0\,:\, S_{\sigma}\,\longrightarrow \, F$.
 This isomorphism induces a $T$--module structure on $F$, or equivalently, a decomposition
\begin{equation}\label{dec1}
F \,=\, \bigoplus_{u \in M} F^{\sigma}_u\, .
\end{equation} 
Restricting to the action of
$T_{\sigma}$ on $\xi_{\sigma}$, we similarly get a
decomposition 
\begin{equation}\label{dec2}
F \,=\, \bigoplus_{[u] \in M_{\sigma}} F^{\sigma}_{[u]}\, . \end{equation}

The decompositions \eqref{dec1} and \eqref{dec2} may depend on the choice of the semi-invariant framing of $\xi_{\sigma}$.
However, for each $\rho \in |\Sigma(1)|$, the subspaces
$$F^{\rho}(i) \,:=\, \bigoplus_{[u]\in M_{\sigma}, u(\rho) \ge i} F^{\sigma}_{[u]}\, ,\quad {\rm where} \; \sigma \; {\rm is \; such \; that}\;
\rho \,\in \,|\Sigma(1)| \bigcap \sigma \, ,$$
 are independent of the choice of
$\sigma$ containing $\rho$ as well as the framing (see \cite{Kly}).

Then $\mb{F}(\xi) $ is defined to be the compatible $\Sigma$--filtered
vector space $\mathcal{F}^{\bullet}\, = \,(F,\, \{ F^{\rho}(i)\})$ on $X$.

\begin{lemma}\label{Ftensor} The functor $\mb{F}\,:\, \mf{Vec}^T(X)\,\longrightarrow \, 
\mf{Cvec}(\Sigma)$ satisfies
$$ \mb{F}(\xi_1) \otimes \mb{F}(\xi_2) \,=\,\mb{F}( \xi_1 \otimes \xi_2) $$
for all $\xi_1, \, \xi_2 \,\in\, \mf{Vec}^T(X)$.
\end{lemma}

\begin{proof} 
Let $V\,=\, \xi_1(x_0)$ and $W\,=\, \xi_2(x_0)$ with $r_i\,=\, \text{dim}(\xi_i(x_0))$. Clearly $V\otimes W \,=\, (\xi_1 \otimes \xi_2) (x_0)$.
Denote,
$$\mb{F}( \xi_1)\,=\,(V, \{V^{\rho}(j)\}) , \; 
\mb{F}( \xi_2) \,=\, (W, \{W^{\rho}(j)\}), \; {\rm and} \;
\mb{F}( \xi_1 \otimes \xi_2) \,=\, (V \otimes W, \{ (V \otimes W)^{\rho}(j) \} ) \, .$$
By \eqref{ten0} and \eqref{ten}, we need to show that
\begin{equation}\label{ten1}
 (V \otimes W)^{\rho}(j) \,=\, \sum_{p+q=j}
 V^{\rho}(p) \otimes W^{\rho}(q) \,. 
\end{equation}

Consider any $\rho \,\in\, |\Sigma(1)|$. Let $\sigma$ be any cone that contains $\rho$. 
 Fix semi-invariant
frames $s^{i}_1,\, \cdots ,\,s^{i}_{r_i}$ of $(\xi_i)_{\sigma}$. Let
$[u^{i}_k] \,\in\, M_{\sigma}$ be the character corresponding to the action of $T_{\sigma}$ on
$s^{i}_k$.
We have compatible decompositions 
$$
V \,=\, \bigoplus_{1 \le k \le r_1} V^{\sigma}_{[u^1_k]} \quad {\rm and} \quad W 
\,=\, \bigoplus_{1 \le l \le r_2} W^{\sigma}_{[u^2_l]}
$$
induced by these frames. Note that $\{ s^1_k \otimes s^2_l \}$ is a
semi-invariant frame of $(\xi_1)_{\sigma} \otimes (\xi_2)_{\sigma}$, which induces a compatible decomposition $$V\otimes W \,=\,
\bigoplus_{[u] \in M_{\sigma}} (V\otimes W)^{\sigma}_{[u]}\, ,$$ where
\begin{equation}\label{Ten} (V\otimes W)^{\sigma}_{[u]} \, =\, \bigoplus_{[u^{1}_k]+[u^{2}_l]=[u]}
 V^{\sigma}_{[u^{1}_k]} \otimes W^{\sigma}_{[u^{2}_l]} 
 \end{equation}
Note that $$V^{\rho}(p) \,= \, \bigoplus_{u^1_k(\rho)\ge p}
V^{\sigma}_{[u^1_k]} \,, ~~~ W^{\rho}(q) \,= \,
\bigoplus_{u^2_l(\rho)\ge q} W^{\sigma}_{[u^2_l]} \, ,~~~ (V\otimes W)^{\rho}(j)\,=\, \bigoplus_{u(\rho) \ge j } (V\otimes W)^{\sigma}_{[u]}\,.$$

 \noindent Therefore, by \eqref{Ten} we get, 
$$V^{\rho}(p) \otimes W^{\rho}(j-p)
 \,\subset\, ( V \otimes W)^{\rho}(j) ~~~~~ {\rm for\, any}~~ p\in \mathbb{Z} \,.$$
 To satisfy \eqref{ten1}, we need to verify that
 $$ ( V \otimes W)^{\rho}(j) \,\subset\, \sum_{p+q=j}
 V^{\rho}(p) \otimes W^{\rho}(q) \,.$$

 Take any $w \,\in \,( V \otimes W)^{\rho}(j)$. Then $w \,=\, \sum
 w_t$, where each $w_t \,\in\, (V\otimes W)^{\sigma}_{[u_t]} $ for some $[u_t] \in M_{\sigma}$ such that $u_t(\rho) \,\ge\,
 j$.
 
It suffices to show that each $w_t \in \sum_{p+q=j}
 V^{\rho}(p) \otimes W^{\rho}(q)$.
 Since $\{ s^{1}_k(x_0) \otimes s^2_l (x_0) \} $ is a basis for $V \otimes W$, we have
$$ w_t \,=\, \sum_{[u^{1}_k]+[u^{2}_l]=[u_t]} a_{kl}^t \, s^{1}_k(x_0) \otimes s^2_l (x_0) \; {\rm where} \; a_{kl}^t \in K \,. $$

 \noindent Note that $[u^{1}_k]+[u^{2}_l]\,=\,[u_t]$ implies $u^1_k(\rho) + u^2_l(\rho)
\,=\, u_t(\rho) \ge j $. Since $u^1_k(\rho)$ 
and $ u^2_l(\rho)$ are integers, there
 exist integers $p$ and $q$ such that $u^1_k(\rho) \ge p$, $u^2_l(\rho) \ge
 q$, and $p+q =j$. It follows that
 $$s^{1}_k(x_0) \otimes s^2_l (x_0) \,\in\, V^{\rho}(p) \otimes W^{\rho}(q) \, , \quad {\rm and}
 \quad
w_t \in \sum_{p+q=j} V^{\rho}(p) \otimes W^{\rho}(q) \,.$$ This completes the proof.
\end{proof}

Since $\mb{F}\,:\, \mf{Vec}^T(X) \,\longrightarrow\, \mf{Cvec}(\Sigma)$ is an equivalence of
additive categories, it is an exact functor. The quasi-inverse
$\mb{K}\,:\, \mf{Cvec}(\Sigma) \,\longrightarrow\, \mf{Vec}^T(X)$ of $\mb{F}$, constructed by
Klyachko in \cite{Kly}, respects direct sums and tensor products.
Being an equivalence of categories, it is also exact and
faithful.

Consider the category $\mf{Cnor}(\Sigma)$ whose objects are functors
$$\mb{M}\,:\, G\text{--}{\rm mod} \,\longrightarrow\, \mathfrak{Cvec}(\Sigma)$$
(see \eqref{not1}) that satisfy properties
F1-F4, and whose morphisms are natural isomorphisms of functors.

\begin{theorem}\label{thm1}
There exists an equivalence of categories between $\mf{Cnor}(\Sigma)$ (defined above) and $\mf{Pbun}_G^{T}(X)$ 
(the category of $T$--equivariant principal $G$--bundles on the toric variety $X$).
\end{theorem}

\begin{proof}
Let
\begin{equation}\label{not3}
\mf{Nor}^{T}(X)
\end{equation}
be the category in \eqref{not2} obtained by substituting $T$ in place of
$\Gamma$. Consider the functor $\mb{F}_{\ast} \,: \, \mf{Nor}^{T}(X) \,\longrightarrow\,
\mf{Cnor}(\Sigma)$ defined by composition with $\mb{F}$,
 $$ \mb{F}_{\ast}(\mb{E}) \,=\, \mb{F}\circ \mb{E}\ \text{for any}\ \mb{E} \,\in\,
\mf{Nor}^{T}(X)\,. $$

Similarly, composition with $\mb{K}$ gives a functor
$\mb{K}_{\ast}\,:\, \mf{Cnor}(\Sigma)\,\longrightarrow\, \mf{Nor}^{T}(X) $,
$$ \mb{K}_{\ast}(\mb{M}) \,=\, \mb{K}\circ \mb{M}\ \text{\rm for any}\ \mb{M} \,\in\,
\mf{Cnor}(\Sigma)\,. $$

It is easily observed from the construction of $\mb{K}$ that
$\mb{F} \circ \mb{K}\,=\, 1_{\mf{Cvec}(\Sigma)} $. Therefore,
$$\mb{F}_{\ast} \circ \mb{K}_{\ast} \,=\, 1_{\mf{Cnor}(\Sigma)} \, .$$
Since $\mb{F}$ and $\mb{K}$ are fully faithful, so are
$\mb{F}_{\ast}$ and $\mb{K}_{\ast}$. Hence, they induce an
equivalence of categories between $\mf{Cnor}(\Sigma)$ and $\mf{Nor}^{T}(X)$.
Then, by Theorem \ref{equiv1},
$\mf{Cnor}(\Sigma)$ and $\mf{Pbun}_G^{T}(X)$ are equivalent categories.
\end{proof}

\section{Equivariant trivialization on affine toric variety} 

The credit for the following result goes to one of our referees.

\begin{theorem}\label{etriv} Let $f\,:\, E_G \,\longrightarrow\, Y$ be a $T$-equivariant principal $G$-bundle where $Y$ is an affine toric variety, defined over $K$, under the action of the torus $T$. Assume $G$ is a reductive algebraic group. Then $E_G$ is equivariantly trivializable. 
\end{theorem}

\begin{proof} $G \times T $ acts on $E_G$ with an open orbit, the preimage under $f$ of the open 
	$T$-orbit in $Y$. Also, the variety $E_G$ is affine, as the morphism $f$ is so. Thus, $E_G$ contains a unique closed $(G\times T)$--orbit, $(G \times T)\cdot e_0$: the image of the closed $T$--orbit, say	$T\cdot y_0$, in $Y$. 
	
	Consider first the case where $y_0$ is fixed by $T$. Then the (scheme-theoretic) isotropy group 
	$(G \times T)_{e_0} \subset G\times T $ intersects $G$ trivially since the latter acts freely on 
	$E_G$. Moreover, $(G \times T)_{e_0}$ is sent onto $T$ by the second projection. Indeed, for any 
	$t \in T$, we have $f(t\cdot e_0) = t \cdot f(e_0) = f(e_0) $, hence $t \cdot e_0 \in G\cdot e_0$. 
	Thus the second projection yields an isomorphism $(G\times T)_{e_0} \cong T$. So there exists a unique homomorphism $\rho\,:\, T \,\longrightarrow\, G $ such that 
	$$ (G\times T)_{e_0} \,=\, \{ (\rho(t),t) \mid t \in T \}\,. $$
	
	In particular, the closed $(G\times T)$--orbit is separable and its isotropy group is linearly reductive. We may now apply \cite[Proposition 8.5]{BR} which yields an equivariant isomorphism
	$$ E_G \cong (G\times T) \times^T F$$
	for some affine $T$--variety $F$. Here $(G\times T) \times^T F$ denotes the associated fiber bundle 
	to the principal $T$--bundle $ G\times T \,\longrightarrow\, (G\times T)/T $ (where $T$ is viewed as a subgroup of $G\times T$ via $t \mapsto (\rho(t), t)$) and the $T$--variety $F$. Thus, $Y = E_G/G \cong F$, and 
	$E_G \cong (G\times T) \times^T Y$. This yields the desired isomorphism 
	$$ E_G \cong G \times^T Y \,, $$ where $T$ acts on $G$ via left multiplication through $\rho$ and on $Y$ via the given action. 
	
	The general proof reduces to the former one as in proof of Corollary 2.3 in \cite{BDP2}.	
\end{proof}

\section{Filtered algebra associated to an equivariant principal bundle}

Henceforth, we assume that $K$ has characteristic zero. Then the reductive group $G$ is linearly
reductive. Let $E_G$ be a $T$--equivariant principal 
$G$--bundle over $X$.

Given any $\mb{E} \in \mf{Nor}^T(X)$, define $\mb{E}_{\sharp} \,\in\,
\mf{Cnor}(\Sigma)$ by
$$
\mb{E_{\sharp}} \,=\, \mb{F} \circ \mb{E} \,.
$$
It is easily checked that
$\mb{E}_{\sharp}$ is faithful. Moreover, it preserves tensor products as a consequence of Lemma \ref{Ftensor}.

\noindent Let $\mb{O}: \mf{Fvec}(\Sigma) \longrightarrow \mf{Vec}$ be the forgetful functor that maps a $\Sigma$--filtered 
vector space to its underlying vector space.
Define ${E}_{\sharp} \,:=\, \mb{O} \circ \mb{E_{\sharp}}$. Note that
$ {E}_{\sharp}(V) \,=\, \mb{E}(V)(x_0)$. It is evident that ${E}_{\sharp}$ preserves tensor products.

 Let $\phi_1,\phi_2 \,:\, V \,\longrightarrow\, W $ be two morphisms of $G$--modules
such that ${E}_{\sharp}(\phi_1) \,=\,{E}_{\sharp} (\phi_2)$. Then $\mb{E}_{\sharp}(\phi_1) =\mb{E}_{\sharp}(\phi_2)$.
By faithfulness of $\mb{E}_{\sharp}$, $\phi_1 = \phi_2$. Therefore, $ {E}_{\sharp} $ is faithful.

\begin{lemma}\label{cext2}
There exists a unique extension of $\mb{E_{\sharp}}$ (respectively,
${E}_{\sharp}$) to a functor $\mb{\overline{E}_{\sharp}}\,:\,
G\text{--}\overline{\rm mod} \,\longrightarrow\, \overline{\mf{Cvec}}(\Sigma)$ (respectively,
${\overline{{E}}_{\sharp}}\,:\, G\text{--}\overline{\rm mod} \,\longrightarrow\, \overline{\mf{Vec}}$)
 that preserves direct limits and tensor products.
\end{lemma}

\begin{proof} 
 It is easily observed that the category of $\Sigma$--filtered vector spaces over $X$ admits direct limits.
For any $\overline{V}$ in $ G$--$\overline{\text{mod}}$, define 
 $\mb{\overline{E}_{\sharp}}(\overline{V})$ (respectively, ${\overline{E}_{\sharp}}(\overline{V})$)
to be the direct limit of $\mb{{E}_{\sharp}}({V})$ (respectively, ${{E}_{\sharp}}({V})$ ) as 
$V$ varies over all finite dimensional $G$--submodules of $\overline{V}$. Note that
direct limit commutes with tensor product. So it follows from
Lemma \ref{Ftensor} that $\mb{\overline{E}_{\sharp}}$ preserves tensor
products.

To understand the compatibility condition, consider the isotypical decomposition 
$$\overline{V} \,=\, \bigoplus_{i \in I} V_i \otimes Hom_G(V_i,\, \overline{V})$$ 
obtained by linear reductivity of $G$. 
Here $I$ denotes the set of isomorphism classes of irreducible $G$--submodules of $\overline{V}$. 
Each of the $V_i$'s is finite dimensional. Since $G$ acts trivially on the module $Hom_G(V_i,\, \overline{V}) $, 
it follows that $\mb{\overline{E}_{\sharp}} ( Hom_G(V_i,\, \overline{V}))$ has trivial $\Sigma$--filtration.
Therefore it may be assigned the trivial compatible decomposition comprising the subspaces
$$ ({\overline{E}_{\sharp}} ( Hom_G(V_i,\, \overline{V})))^{\sigma}_{[u]} \,=\, \left\{ 
 \begin{array}{ ll} {\overline{E}_{\sharp}} ( Hom_G(V_i,\, \overline{V})) & {\rm if} \, [u] \,=\, 0 \,, \\
 0 & {\rm otherwise} \,. \end{array} \right. $$
Now a choice of compatible decomposition for each $\mb{{E}}_{\sharp}(V_i)$ determines a compatible decomposition for 
$\mb{\overline{E}}_{\sharp}(\overline{V}) $.
 
The construction of $\mb{\overline{E}_{\sharp}}(\overline{\phi})$ for a
morphism $\overline{\phi}\,:\, \overline{U} \,\longrightarrow\, \overline{V}$ is similar to Lemma \ref{cext}.
\end{proof}

\begin{lemma}\label{f3f4} The functors
 $\mb{\overline{E}_{\sharp}}$ (respectively, ${\overline{E}_{\sharp}}$) satisfy properties {\rm F2, F3} and {\rm F4}.
\end{lemma}

\begin{proof} Since direct limit commutes with tensor product,
$\mb{\overline{E}_{\sharp}}$ satisfies F2 and F3.

Suppose there exist morphisms $\overline{\phi}_j\,:\, \overline{U} \,\longrightarrow\, \overline{V}
$, $j=1,\,2$, such that $\mb{\overline{E}_{\sharp}}( \overline{\phi}_1 ) \,=\, 
\mb{\overline{E}_{\sharp}}( \overline{\phi}_2 ) $. To prove
$\mb{\overline{E}_{\sharp}}$ is faithful, it is enough to show that $
\overline{\phi}_1 \,=\, \overline{\phi}_2 $.
Consider any element $u \,\in\, \overline{U}$. Then there exists a finite
dimensional $G$--submodule $U$ of $\overline{U}$ such that $u \in U$.
Let $i_U\,:\, U \,\longrightarrow \,\overline{U}$ be the inclusion map. Let $\phi_j \,=\,
\overline{\phi}_j \circ i_U$. Then
\begin{equation}\label{phi12}
\mb{\overline{E}_{\sharp}}(\phi_1) \,=\,
\mb{\overline{E}_{\sharp}}(\phi_2)\,.\end{equation}

There exists a finite dimensional $G$--module $V \,\subset \,\overline{V}$
such that $\phi_j(U) \,\subset\, V$ for each $j$. Let $i_V\,:\, V \,\longrightarrow\,
\overline{V}$ be the inclusion map. Let $\psi_j\,:\, U \,\longrightarrow\, V$ be the unique
map such that ${\phi}_j \,=\, i_V \circ \psi_j$.
So, by \eqref{phi12},
\begin{equation}\label{phi122}
\mb{\overline{E}_{\sharp}}(i_V) \circ \mb{E_{\sharp}}( \psi_1) \,=\,
\mb{\overline{E}_{\sharp}}(i_V) \circ \mb{E_{\sharp}}( \psi_2) \,.
\end{equation}
Since $\mb{E_{\sharp}}$ is faithful, $\mb{\overline{E}_{\sharp}}(i_V)$
is a direct limit of inclusion maps. Therefore,
$\mb{\overline{E}_{\sharp}}(i_V)$ is injective. Hence it follows from
\eqref{phi122} that
\begin{equation}\label{phi123}
\mb{E_{\sharp}}(\psi_1) \,=\,
 \mb{E_{\sharp}}( \psi_2) \,.
\end{equation}

 Hence, by faithfulness of $\mb{E_{\sharp}}$ and \eqref{phi123},
we have $ \psi_1 \,=\, \psi_2 $. It follows that $\phi_1=
\phi_2$, and hence $\overline{\phi}_1(u)\,=\, \overline{\phi}_2(u)$. Since $u $
is arbitrary, $\overline{\phi}_1\,=\, \overline{\phi}_2$.

The proof for ${\overline{E}_{\sharp}}$ is similar. \end{proof}

\begin{lemma}\label{Tsch} $\mb{\overline{E}_{\sharp}}$ (respectively, ${\overline{E}_{\sharp}}$) defines a functor from
 affine $G$--schemes to
 $\Sigma$--filtered algebras (respectively, K-algebras). We denote this functor by $\mb{\overline{E}_{\sharp}}$ 
(respectively, ${\overline{E}_{\sharp}}$) as well.
\end{lemma}

\begin{proof} Suppose $A$ is a finitely generated $K$--algebra on which $G$ acts. Following Nori, we view $A$ as a
locally finite $G$--module, and its multiplication as a
morphism $$m_A\,:\, A \otimes A \,\longrightarrow\, A$$ in $G\text{--}\overline{\rm mod}$. 

Note that $ {\overline{E}_{\sharp}}( A \otimes A)\,=\, {\overline{E}_{\sharp}}(A) \otimes {\overline{E}_{\sharp}}(A)$. 
Therefore we have a morphism 
$$
{\overline{E}_{\sharp}}(m_A)\,:\, {\overline{E}_{\sharp}}(A) \otimes {\overline{E}_{\sharp}}(A) 
\,\longrightarrow \,{\overline{E}_{\sharp}}(A) \, .
$$
Since $m_{A}$ is nontrivial and $ {\overline{E}_{\sharp}}$ is faithful, $ {\overline{E}_{\sharp}}(m_A)$ defines a
nontrivial multiplication on ${\overline{E}_{\sharp}}(A)$. This is commutative and associative, since $ {\overline{E}_{\sharp}}$ satisfies property F3.

 For any two finite dimensional $G$--submodules $V$ and $W$ of $A$, by \eqref{ten} we have
$${{E}_{\sharp}}(V)^{\rho }(i) \otimes {{E}_{\sharp}}(W)^{\rho}(j)
\,\subseteq\, ({{E}_{\sharp}}(V) \otimes {E}_{\sharp}(W))_{\dagger}^{\rho} (i+j) \,.$$

Since $\mb{{E}_{\sharp}}( V \otimes W) \,=\, \mb{{E}_{\sharp}}(V) \otimes \mb{{E}_{\sharp}}(W) $,
$$ ({{E}_{\sharp}}(V) \otimes {E}_{\sharp}(W))_{\dagger}^{\rho} (i+j) \,=\, ({{E}_{\sharp}}(V \otimes W))^{\rho} (i+j) \,.$$ 

Then applying Lemma \ref{cext2}, we have 
\begin{equation}\label{hit1}
{\overline{E}_{\sharp}}(A)^{\rho }(i) \otimes {\overline{E}_{\sharp}}(A)^{\rho}(j)
\,\subseteq\, ({\overline{E}_{\sharp}}(A \otimes A))^{\rho} (i+j) \,.\end{equation}

As $\mb{\overline{E}_{\sharp}}(m_A)$
is a morphism of $\Sigma$--filtered vector spaces, we have
\begin{equation}\label{hit2}
\mb{\overline{E}_{\sharp}}(m_A) ({\overline{E}_{\sharp}}(A \otimes A))^{\rho} (i+j)
 \,\subseteq\, {\overline{E}_{\sharp}}(A)^{\rho}(i+j) \,. \end{equation}
By \eqref{hit1} and \eqref{hit2},
$$ \mb{\overline{E}_{\sharp}}(m_A) ({\overline{E}_{\sharp}}(A)^{\rho }(i) \otimes {\overline{E}_{\sharp}}
(A)^{\rho}(j))\, \subseteq\, {\overline{E}_{\sharp}}(A)^{\rho}(i+j) \,.$$
Thus $\mb{\overline{E}_{\sharp}}(A)$ is a $\Sigma$--filtered algebra.

Now suppose that $\phi\,:\, A \,\longrightarrow\, B$ is a morphism of $G$--algebras, so that
$$ \phi \circ m_A \,=\, m_B \circ (\phi \otimes \phi).$$
By functoriality,
$$ \mb{\overline{E}_{\sharp}}( \phi) \circ \mb{\overline{E}_{\sharp}}( m_A)\, = \,\mb{\overline{E}_{\sharp}}( m_B) \circ 
(\mb{\overline{E}_{\sharp}}(\phi) \otimes \mb{\overline{E}_{\sharp}}( \phi))\,.$$ 
Thus $ \mb{\overline{E}_{\sharp}}(\phi)$ is a morphism of $\Sigma$--filtered algebras.
\end{proof}

\begin{lemma}\label{fiber}
The algebra ${\overline{E}}_{\sharp}(K[G])$ is the
algebra of regular functions of the fiber ${\overline{E}}(K[G])(x_0) $ of the principal bundle ${\overline{E}} (K[G])$.
Moreover, ${\overline{E}}_{\sharp}(K[G])$ 
admits a $G$--action which makes it equivariantly isomorphic to $K[G]$.
\end{lemma}

\begin{proof}
Let $\mb{\overline{O}}$ be the forgetful functor that takes a $\Sigma$--filtered algebra to its underlying $K$--algebra. 
Then
\begin{equation}\label{ua}
{\overline{E}_{\sharp}}\,=\, \mb{\overline{O}} \circ \mb{\overline{E}}_{\sharp}\,. \end{equation} 
Indeed, this follows from ${{E}_{\sharp}} \,=\,\mb{{O}} \circ \mb{{E}}_{\sharp}$ combined with Lemma \ref{f3f4}
and the uniqueness of the extensions of ${E}_{\sharp}$ and $ \mb{{E}}_{\sharp}$ to $G\text{--}\overline{\rm mod}$.

Let $q \,:\, \mc{O}_{X} \,\longrightarrow\, K$ be the evaluation map corresponding to the closed point $x_0 \in X$. 
Using $q$, to any $\mc{O}_{X}$--module
 $M$ we may associate a $K$--vector space $M \otimes_{\mc{O}_{X}} K$. Now recall that, for any $V 
\,\in\, G$--mod,
 $${E_{\sharp}}(V) \,=\, \mb{E}(V) \otimes_{\mc{O}_X} K \,.$$

 Then by the uniqueness of extensions we have
$${\overline{E}_{\sharp}}(K[G]) \,= \,\mb{\overline{E}}(K[G]) \otimes_{\mc{O}_X} K \,.$$
It is then clear that $ {\overline{E}_{\sharp}}(K[G])$ is the algebra of regular functions of the fiber at $x_0$ of
the principal bundle ${\overline{E}}(K[G])$. Note that, by \eqref{ua}, this $\overline{E_{\sharp}}(K[G])$ is the underlying
 algebra of $\mb{\overline{E}}_{\sharp}(K[G])$. This completes the proof of the first part of the lemma.

Next note that there is a natural $G$--action on the principal bundle ${\overline{E}}(K[G])$ which
is free and transitive on each fiber. 
This yields the required $G$--action on ${\overline{E}}_{\sharp}(K[G])$ by the first part of the lemma. 
\end{proof}

\begin{lemma}\label{compalg} Let $X$ be a toric variety.
Then the $\Sigma$--filtered algebra $\mb{\overline{E}}_{\sharp}(K[G])$ is compatible.
\end{lemma}

\begin{proof} By Remark \ref{remcomp}, it is enough to concentrate on a maximal cone 
$\sigma$.
It follows from Theorem \ref{etriv} (cf. \cite{BDP2}) that $E_G$ admits a section $s$ over 
$X_{\sigma}$, such that 
\begin{itemize}
\item $t s(x) \,=\, s(t x) \rho_s(t)$ for every $x \,\in\, X_{\sigma}$ and $t \,\in\, T,$ where
$\rho_{s} \,:\, T\,\longrightarrow\, G$ is a group homomorphism. 
\end{itemize}

Then for any locally finite $G$--module $V$, the homomorphism $\rho_s$ and the given action of $G$ induce a $T$--action on $V$
which we denote by $\rho_s$ again without confusion. An eigenvector $v$ of this action with weight $\chi(t)$ gives 
rise to a semi-invariant section $[(s(x),v)]$ of $\overline{\mb{E}}(V)$ with the same weight. Such sections, corresponding to a choice of an
eigen-basis of $V$, induce a compatible $T$--decomposition of $\mb{\overline{E}}_{\sharp}(V)$. Moreover, this decomposition does not
depend on the choice of the eigen-basis.

Now consider the $G$--module $K[G]$. The action of $T$ on $K[G]$ induced by $\rho_s$ satisfies the condition
that
$$ \rho_s(t) f(\cdot) \,=\, f( \cdot \, \rho_s(t))\,. $$
 It follows that if $f_1,\, f_2 \,\in\, K[G]$ are T-eigenvectors with weights $\chi_1(t)$ and $\chi_2(t)$ respectively, then 
the product $f_1 f_2$ is a T-eigenvector with weight $\chi_1(t) \chi_2(t)$. This implies that the compatible $T$--decomposition on 
on $\mb{\overline{E}}_{\sharp}(K[G])$ respects the multiplication of the algebra $K[G]$ (cf. \eqref{ca3}).
\end{proof}

\begin{lemma}\label{commute} For every $n$--dimensional cone $\sigma$, an action of $T$ on ${\overline{E}_{\sharp}}(K[G])$
 which is compatible with the $\Sigma$--filtration, commutes with the action
of $G$.
\end{lemma}

\begin{proof}
We revisit the $G$--action on ${\overline{E}_{\sharp}} (K[G])$. Recall the multiplication map $a
\,:\, G \times G' \,\longrightarrow\, G$ of $G$ defined in \eqref{trivac1}.
Let $ a^* \,:\, K[G] \,\longrightarrow\, K[G] \otimes K[G']$ be
the algebra morphism corresponding to $a$. 
Then we have a map
\begin{equation}\label{eqgact2}
\mb{\overline{E}_{\sharp}}(a^*) \,:\, \mb{\overline{E}_{\sharp}}(K[G])\,\longrightarrow\,\mb{\overline{E}_{\sharp}}(K[G])
\otimes\mb{\overline{E}_{\sharp}}(K[G'])\, . \end{equation}
of $\Sigma$--filtered algebras. Note that the underlying algebra ${\overline{E}_{\sharp}} (K[G']) $ of $ \mb{\overline{E}_{\sharp}} (K[G'])$ is $ K[G']$. 
This follows from the fact observed in Section
\ref{NC} that ${\overline{E}}(K[G']) \,=\, X \times G'$, and the first part of Lemma \ref{fiber}. Thus the map $ \mb{\overline{E}_{\sharp}}(a^*)$ induces an action of $G$ on 
 ${\rm Spec} ({\overline{E}_{\sharp}} (K[G]) ) \,=\, {\overline{E}} (K[G]) (x_0) $, which agrees with the action of $G$ given by
$\overline{E}(a)$ (cf. \eqref{trivac2}).
 
It follows from property F3(c) that the bundle ${\overline{E}}(K[G'])$ has trivial $T$--action
 for any $\mb{E} \,\in\, \mf{Nor}^T(X)$.
Therefore, the $\Sigma$--filtration on $\mb{\overline{E}_{\sharp}}(K[G'])$ satisfies
\begin{equation}\label{trivfil}
\mb{\overline{E}_{\sharp}}(K[G'])^{\rho}(i)\,= \,\left\{ \begin{array}{ll} K[G'] &
 \text{if}\, i \le 0 \\
0 & \text{if} \, i> 0 \\
\end{array} \right. \end{equation}
for every $\rho \in |\Sigma(1)|$. Since the filtrations are decreasing, it follows (cf. \eqref{ten}) that for every $\rho$,
\begin{equation}\label{kgp}
\left( \mb{\overline{E}_{\sharp}}(K[G]) \otimes \mb{\overline{E}_{\sharp}}(K[G']) \right)^{\rho} (i)
\,=\, \mb{\overline{E}_{\sharp}}(K[G])^{\rho}(i) \otimes
K[G'] \,. \end{equation}
As $\mb{\overline{E}_{\sharp}}(a^*) $ respects the $\Sigma$--filtration, we have
$$\mb{\overline{E}_{\sharp}}(a^*) (\mb{\overline{E}_{\sharp}}(K[G])^{\rho}(i) ) \subset \mb{\overline{E}_{\sharp}}(K[G])^{\rho}(i)
\otimes K[G'] \,. $$

Let $\sigma$ be an $n$--dimensional cone. Note that $T_{\sigma} \,=\,T$.
Fix a decomposition (equivalently, a compatible $T$--action)
\begin{equation}\label{eqtact}
\mb{\overline{E}_{\sharp}}(K[G]) \,=\, \bigoplus_{\chi \in M} \mb{\overline{E}_{\sharp}}(K[G])^{\sigma}_{\chi} \, , \end{equation} such that for any
$\rho \in \sigma \bigcap \Sigma(1)$,
$$\mb{\overline{E}_{\sharp}}(K[G])^{\rho}(i) \,=\, \sum\limits_{\chi(\rho)\ge i } \mb{\overline{E}_{\sharp}}(K[G])^{\sigma}_{\chi} \,.$$
We will show that
\begin{equation}\label{eqfa}
 \mb{\overline{E}_{\sharp}}(a^*) (\mb{\overline{E}_{\sharp}}(K[G])^{\sigma}_{\chi} )
\,\subseteq\, \mb{\overline{E}_{\sharp}}(K[G])^{\sigma}_{\chi}
\otimes K[G'] \end{equation} for every $\chi$.

Suppose $f \in \mb{\overline{E}_{\sharp}}(K[G])^{\sigma}_{\chi} $. 
Since $K[G] \otimes K[G']$ is locally finite, we may
write $\mb{\overline{E}_{\sharp}}(a^*) (f)$ as a finite sum
 \begin{equation}\label{eqf}
\mb{\overline{E}_{\sharp}}(a^*) (f) \,=\, \sum f_j \otimes b_j, \end{equation} where
$f_j \,\in\, \mb{\overline{E}_{\sharp}}(K[G])^{\sigma}_{\chi_j} $ and $b_j \in K[G'] $.
Note that $$ f \in \mb{\overline{E}_{\sharp}}(K[G])^{\rho} (\chi(\rho)) \; \text{for} \; \text{every} \; \rho \in \sigma \bigcap \Sigma(1) \,. $$
Since $\mb{\overline{E}_{\sharp}}(a^*)$ preserves the $\Sigma$--filtration, we must have $\chi_j(\rho) \ge \chi(\rho) $
for every $\rho \,\in\, \sigma \bigcap \Sigma(1)$.

Suppose, if possible, $\chi_{j_0} \,\neq\, \chi$ for some value $j_0$ of $j$. Then 
since $\sigma \bigcap \Sigma(1) $ spans $N \otimes \mathbb{R}$, there exists $\rho_0 
\in \sigma \bigcap \Sigma(1) $ such that $ \chi_{j_0}(\rho_0) > \chi (\rho_0)$.

Given any $h \in G'$, consider the $G$--map $$\phi_h\,:\, G \,\longrightarrow\, G 
\times G'$$ defines by $g\,\longmapsto\, (g,\, h)$. The induced map
$$\phi_h^* \,:\, K[G] \otimes K[G'] \,\longrightarrow \,K[G]$$ satisfies 
$\phi_h^*(x \otimes y) = y(h) \,x$. Identifying $K[G] $ with
$K[G] \otimes_K K$, we can write
$$\phi_h^* \,=\, {\rm id} \otimes ev_h$$ where $$ev_h: K[G'] \longrightarrow 
K$$ is the evaluation map at $h$.
Therefore, $$ \mb{\overline{E}_{\sharp}} (\phi_h^* )\, =\, {\rm id} \otimes \mb{\overline{E}_{\sharp}}(ev_h )\,. $$

Note that $\mb{E}(K) $ is the trivial line bundle with trivial $T$--action by property F3(c).
Therefore, $\mb{\overline{E}_{\sharp}} (K) \,=\, K$.
Moreover, for any two $G$--algebras $A,\, B$ with trivial $G$--action, and any homomorphism $\theta
\,:\, A \,\longrightarrow\, B$,
$\mb{\overline{E}_{\sharp}} (\theta) \,=\, \theta$. Hence, $ \mb{\overline{E}_{\sharp}} ({ev_h})
\,= \,ev_h $. Thus we have,
$$\mb{\overline{E}_{\sharp}} (\phi_h^* ) \,=\, {\rm id} \otimes {ev_h}:
\mb{\overline{E}_{\sharp}}(K[G]) \otimes K[G'] \,\longrightarrow\, \mb{\overline{E}_{\sharp}}(K[G]) \,. $$
Hence, $$ \mb{\overline{E}_{\sharp}} (\phi_h^* ) \, (\sum f_j \otimes b_j)
\,=\, \sum b_j(h)f_j \,.$$
Then using \eqref{eqf} we have
 $$ \mb{\overline{E}_{\sharp}} (\phi_h^* \circ a^* ) (f) \,=\, \sum b_j(h)f_j \,.$$
 Since $b_{j_0} \neq 0$, there exists $h_0$ such that $b_{j_0}(h_0) \neq
 0$. Let $i_0 = \chi(\rho_0)$. Then $$b_{j_0}(h_0)\,f_{j_0} \in
 \mb{\overline{E}_{\sharp}}(K[G])^{\rho_0}(i_1)\; \text{where}\; i_1 \,=\, \chi_{j_0}(\rho_0)\, > \,i_0
 \,.$$
 Note that the composition of maps
 $$ (\phi_{h^{-1}_0}^* \circ a^* ) \circ (\phi_{h_0}^* \circ a^*)= \rm{id}\,. $$
 Since $ \mb{\overline{E}_{\sharp}}(\phi_{h^{-1}_0}^* \circ a^* )$ is also
 filtration preserving,
 $$ \mb{\overline{E}_{\sharp}}(\phi_{h^{-1}_0}^* \circ a^* )(b_{j_0}(h_0)\,f_{j_0}) \,\in\,
 \mb{\overline{E}_{\sharp}}(K[G])^{\rho_0}(i_1) \,.$$
 Therefore,$$ \mb{\overline{E}_{\sharp}}(\phi_{h^{-1}_0}^* \circ a^* )(b_{j_0}(h_0)\,f_{j_0})
 \notin \mb{\overline{E}_{\sharp}}(K[G])^{\sigma}_{\chi} \, $$
 unless $ \mb{\overline{E}_{\sharp}}(\phi_{h^{-1}_0}^* \circ a^*
 )(b_{j_0}(h_0)\,f_{j_0})\,=\, 0$.
 But $$ \mb{\overline{E}_{\sharp}}(\phi_{h^{-1}_0}^* \circ a^* )
 \circ \mb{\overline{E}_{\sharp}}(\phi_{h^{-1}_0}^* \circ a^* )(f)
 \,=\, f \in \mb{\overline{E}_{\sharp}}(K[G])^{\sigma}_{\chi}\,. $$
 Therefore, $$\mb{\overline{E}_{\sharp}}(\phi_{h^{-1}_0}^* \circ a^* )(b_{j_0}(h_0)\,f_{j_0})= 0 \,.$$
 This is a contradiction since $\mb{\overline{E}_{\sharp}}(\phi_{h^{-1}_0}^* \circ a^*)$
 is an isomorphism. Thus no such $j_0$ exists. Therefore, using \eqref{eqf}, we obtain \eqref{eqfa}.
 This implies that the actions of $G$ and $T$ on $\mb{\overline{E}_{\sharp}}
(K[G])$, induced by \eqref{eqgact2} and \eqref{eqtact} respectively, commute. The lemma follows.
\end{proof}

We may associate to any equivariant principal bundle $E_G$, the
compatible $\Sigma$--filtered algebra $\mb{\overline{E}_{\sharp}}
(K[G])$, where $\mb{E}\,= \,\mb{N}^{T}_0(E_G)$.

\section{Correspondence between equivariant $G$--bundles and $\Sigma$--filtered algebras}

Let $X$ be a toric variety defined over an algebraically closed field $K$ of characteristic $0$.
 Assume that every maximal cone in the fan $\Sigma$ of $X$ is of top dimension. We 
note that this assumption is always satisfied when $X$ is a complete toric variety.

\begin{defn}\label{def1}
Let $\mf{Calg}_G(\Sigma)$ be the category whose objects are compatible
$\Sigma$--filtered $K$--algebras $B$ such that
\begin{itemize}
\item $B$ admits a $G$--action with respect to which it is $G$--equivariantly isomorphic to the algebra $K[G]$,

\item For every top dimensional cone in the fan, $B$ admits a compatible
action of $T$ that commutes with the action of $G$ on $B$.
\end{itemize}

The morphisms of $\mf{Calg}_G(\Sigma)$ are $G$--equivariant isomorphisms of compatible
$\Sigma$--filtered $K$--algebras.
\end{defn}

\begin{lemma}\label{leml}
The association $\mb{E} \,\longmapsto\,
\mb{\overline{E}_{\sharp}}(K[G])$ 
induces a functor $$\mb{A}\,:\,\mf{Nor}^T(X)
\,\longrightarrow\, \mf{Calg}_G(\Sigma)\, ,$$
where $\mf{Nor}^T(X)$ and $\mf{Calg}_G(\Sigma)$ are as
in \eqref{not3} and Definition \ref{def1} respectively.
\end{lemma}

\begin{proof} It follows from Section 4 that $\mb{\overline{E}_{\sharp}}(K[G])$ is an 
object in $\mf{Calg}_G(\Sigma)$.

Let $\Psi: \mb{E}^1 \longrightarrow \mb{E}^2$ be a morphism in $\mf{Nor}^T(X)$.
 For any morphism $f\,:\,
V\,\longrightarrow\,W$ in $G$--mod, we have a commuting diagram.
$$
\begin{CD}
\mb{E^{1}_{\sharp}}(V) @> \mb{F}\circ\Psi(V) >> \mb{E^{2}_{\sharp}}(V) \\
@V \mb{E^{2}_{\sharp}}(f)VV @V{\mb{E^{2}_{\sharp}}(f)}VV \\
 \mb{E^{1}_{\sharp}}(W) @>\mb{F}\circ \Psi(W)>> \mb{E^{2}_{\sharp}}(W)
\end{CD}
$$

So the direct limit of the morphisms $\mb{F}\circ\Psi(V)$ exists
as $V$ runs over all finite dimensional $G$--submodules of
$K[G]$. We denote this limit by $\mb{A}(\Psi)\,:\,
\mb{\overline{E}^1_{\sharp} }(K[G])\,\longrightarrow\,\mb{\overline{E}^2_{\sharp}
}(K[G])$.

Let $G'$ be a copy of $G$ with trivial $G$--action as in
\eqref{trivac1}. We will denote the limit of $\mb{F} \circ \Psi(V)$ as $V$ varies
over all finite dimensional $G$--submodules of $K[G']$ by
$\mb{A}'(\Psi)$. Note that $\mb{\overline{E}^{j}_{\sharp}}(K[G']) $ is the algebra
$K[G'] $ with the trivial filtration \eqref{trivfil} and $\mb{A}'(\Psi) = \rm{id}$ using property F3(c).

Since the $\mb{F}\circ \Psi(V)$'s are
morphisms of filtered vector spaces and the filtration on
$\mb{\overline{E}^{j}_{\sharp}}(K[G])$ is the direct limit of the
filtrations on $ \mb{\overline{E}^{j}_{\sharp}}(V)$, it follows that $\mb{A}(\Psi)$ is
a morphism of $\Sigma$--filtered vector spaces.

Since $\Psi$, by definition, respects F1-F4, it follows that ${\mb{A}}(\Psi)$ is a morphism of algebras.
Regard the action of $G$ on $K[G]$ as a morphism of algebras $a^*
\,:\, K[G]\,\longrightarrow\,K[G] \otimes K[G']$. By \eqref{kgp}, the $G$--action on
$\mb{\overline{E}^{j}_{\sharp}}(K[G])$ is given by
$$ \mb{\overline{E}^{j}_{\sharp}}(a^*)\,:\, \mb{\overline{E}^{j}_{\sharp}}(K[G])
\,\longrightarrow\, \mb{\overline{E}^{j}_{\sharp}}(K[G])
\otimes K[G'] \,.$$ Again, using functoriality, we have a
commutative diagram
$$
\begin{CD}
\mb{\overline{E}^{1}_{\sharp}}(K[G]) @> \mb{\overline{E}^{1}_{\sharp}}(a^*) >> \mb{\overline{E}^{1}_{\sharp}}(K[G])
\otimes K[G']\\
@V \mb{A}(\Psi) VV @V \mb{A}(\Psi) \otimes \rm{id} VV \\
\mb{\overline{E}^{2}_{\sharp}}(K[G]) @> \mb{\overline{E}^{2}_{\sharp}}(a^*)
>> \mb{\overline{E}^{2}_{\sharp}}(K[G]) \otimes K[G']
\end{CD}
$$
This shows that $\mb{A}(\Psi)$ is $G$--equivariant.
\end{proof}

\begin{lemma}\label{faithful} The functor $\mb{A} \,:\,\mf{Nor}^T(X)
\,\longrightarrow \,\mf{Calg}_G(\Sigma)$ is faithful.
\end{lemma}

 \begin{proof} Let $\Psi_j \,:\, \mb{E}^1 \,\longrightarrow\, \mb{E}^2$, $j\,=\, 1,2,$ be two morphisms.
 By Lemma \ref{fiber}, the underlying algebra of $\mb{A}(\mb{E}^j)$ is the coordinate algebra
$$ K[ {\overline{E}}^j(K[G]) (x_0)] = K[ \mb{N}^{T}_1(\mb{E}^j ) (x_0)]$$
 If $\mb{A}(\Psi_1) \,=\, \mb{A}(\Psi_2)$, then
$$\mb{N}^{T}_1 (\Psi_1)|_{x_0} \,=\, \mb{N}^{T}_1 (\Psi_1)|_{x_0}\,. $$

Now, by $T$--equivariance, $\mb{N}^{T}_1 (\Psi_1)$ and
$\mb{N}^{T}_1 (\Psi_2)$ must agree over the open $T$--orbit.
Hence, they must agree over $X$. Since
$\mb{N}_1^{T}$ is faithful, we conclude that $\Psi_1 \,=\, \Psi_2$.
\end{proof}

\begin{lemma}\label{inter} Consider an object $B$ in $\mf{Calg}_G(\Sigma)$. Fix
an embedding $\theta$ of $G$ in ${\rm GL}(V)$. Then the
 $\Sigma$--filtration on $B$ induces a compatible $\Sigma$--filtration on 
 $(B \otimes K[V])^G$.
\end{lemma}

\begin{proof}
Suppose $\sigma$ is a maximal cone. Consider an action of 
$T_{\sigma}=T$ on $B$ which is compatible with the $\Sigma|_{\sigma}$--filtration. Let
$$
B \,=\, \bigoplus_{u \in M_{\sigma}} B^{\sigma}_u \,.
$$
 be the corresponding isotypical decomposition. 
 
 We first claim that 
 \begin{equation}\label{decom2} 
 (B \otimes K[V])^G \,=\, \bigoplus_{u \in M_{\sigma}} (B^{\sigma}_u \otimes K[V])^G \, . \end{equation} 
 
 Let $x \in (B \otimes K[V])^G$. We may write $x$ uniquely as a finite sum 
 \begin{equation}\label{eqnx} 
x\,=\, \sum x_u\, ,
 \end{equation}
 where $x_u \,\in\, B^{\sigma}_u \otimes K[V]$. 
 Since the actions of $T$ and $G$ on $B$ commute, $B^{\sigma}_u$ is $G$--invariant. 
 This implies that $g x_u \,\in\, B^{\sigma}_u \otimes K[V]$ for any $g \,\in\, G$. 
 Since $gx=x$, we have 
$$
x\,=\, \sum g x_u
$$
 which is another decomposition of $x$ with components in $B^{\sigma}_u \otimes K[V]$.
 By the uniqueness of \eqref{eqnx}, we must have $gx_u = x_u$ for all $u$ and $g$.
 This means that $x_u \,\in\, (B^{\sigma}_u \otimes K[V])^G $. Hence, 
$$ (B \otimes K[V])^G \,\subseteq\, \bigoplus_{u \in M_{\sigma}} (B^{\sigma}_u \otimes K[V])^G \, .$$
 Clearly
 $$(B \otimes K[V])^G \,\supseteq\, \bigoplus_{u \in M_{\sigma}} (B^{\sigma}_u \otimes K[V])^G \, . $$
 Hence \eqref{decom2} follows. Using it we conclude that 
 \begin{equation}\label{decom3b}
 ((B \otimes K[V])^G)^{\sigma}_u\,=\, (B^{\sigma}_u \otimes K[V])^G \,. \end{equation}
 
 By the compatibility of the $\Sigma$--filtration on $B$, the decomposition 
 $$ B^{\rho}(i) = \bigoplus_{ u(\rho) \ge i} B^{\sigma}_u $$
 is independent of the choice of $\sigma$ such that $\rho \in \sigma$.
 By \eqref{decom3b}, 
 \begin{equation}\label{decomdef}
((B \otimes K[V])^G)^{\rho}(i) \,=\, \bigoplus_{ u(\rho)\ge i} (B^{\sigma}_u \otimes K[V])^G \,.
\end{equation}
 To show the independence of \eqref{decomdef} of the choice of maximal cone 
 $\sigma$, we want to show that 
 \begin{equation} \label{decomf}
((B \otimes K[V])^G)^{\rho}(i) \,= \,(B^{\rho}(i) \otimes K[V])^G \, .
\end{equation}

If $v(\rho) \ge i$, then 
$$(B^{\sigma}_v \otimes K[V])^G \,\subseteq\,(( \bigoplus_{u(\rho) \ge i } B^{\sigma}_u) \otimes K[V])^G \,=\,(B^{\rho}(i) \otimes K[V])^G \,. $$ Therefore, by \eqref{decomdef}, 
\begin{equation}\label{con1}
((B \otimes K[V])^G)^{\rho}(i) \,\subseteq\, (B^{\rho}(i) \otimes K[V])^G \, . 
\end{equation}

On the other hand, suppose $x \in (B^{\rho}(i) \otimes K[V])^G$. Then 
$x$ admits a unique decomposition $$ x= \sum_{u(\rho) \ge i} x_u$$ where 
$x_u \,\in\, B^{\sigma}_u \otimes K[V]$. Then by using the $G$--invariance of $x$ and the uniqueness of the 
decomposition as before, we conclude that $g x_u= x_u$ for all $g \in G$.
Hence $x_u \in (B^{\sigma}_u \otimes K[V])^G$, and consequently,
$$x \,\in\, \bigoplus_{ u(\rho)\ge i} (B^{\sigma}_u \otimes K[V])^G = ((B \otimes K[V])^G)^{\rho}(i) \,.$$
Hence, \begin{equation}\label{con2}
((B \otimes K[V])^G)^{\rho}(i) \,\supseteq \, (B^{\rho}(i) \otimes K[V])^G \, .
\end{equation}
By \eqref{con1} and \eqref{con2}, equation \eqref{decomf} holds, concluding the proof.
\end{proof}

\begin{lemma}\label{surj} Assume all maximal cones in $\Sigma$ are top dimensional. Then the 
functor $\mb{A} : \mf{Nor}^T(X) \longrightarrow \mf{Calg}_G(\Sigma) $ is essentially surjective.
\end{lemma}

\begin{proof} Consider an object $B$ in $\mf{Calg}_G(\Sigma)$.
Consider a top-dimensional cone $\sigma$ in $\Sigma$. Note that $T_{\sigma} =T$.

Fix a $T$--action on $B$ which is compatible with the $\Sigma$--filtration and 
commutes with the $G$--action on $B$. Define $E^{\sigma}_G \,=\, X_{\sigma} \times 
{\rm Spec}(B)$. With the induced actions of $T$ and $G$, this $E^{\sigma}_G$ is a 
$T$--equivariant principal $G$--bundle over $X_{\sigma}$.

Fix a closed point $y_0$ in ${\rm Spec}(B)$. Recall the closed point $x_0 \in O$ used in the construction of $\mb{F}$. 
Let $e=(x_0,y_0) \in E^{\sigma}_G$. The $T$--action on ${\rm Spec}(B)$ may be represented by a
homomorphism $$\rho_{\sigma}: T \longrightarrow G, \quad {\rm defined\; by} \quad t y_0 \,= \, y_0 \cdot \rho_{\sigma}(t) \;
{\rm for \; any}\; t \,\in\, T \, .$$
We claim that for any two top-dimensional cones $\sigma$ and $\tau$, the functions
$$\rho_{\sigma} \rho_{\tau}^{-1} \,:\, T \,\longrightarrow\, G $$ extend to regular $G$--valued functions over $X_{\sigma} \bigcap X_{\tau}$ under the
standard identification of $T$ with the open orbit $O$ in
$X$.

Fix an embedding $\theta$ of $G$ in ${\rm GL}(V)$. Let $E^{\sigma} = E^{\sigma}_G \times^G V$ be the associated
$T$--equivariant vector bundle.
The actions of $T$ and $G$ on ${\rm Spec}(B)\,=\,E^{\sigma}_G (x_0)$ commute and hence induce a
$T$--action on $E^{\sigma}(x_0) \,=\, {\rm Spec}(B)\times^G V $. Using a specific isomorphism 
\begin{equation}\label{siso}
{\rm Spec}(B)\times^G V \,\cong\, V \quad {\rm induced \, by \, the\, rule}\quad [(y_0,v)] \longmapsto v \,,\end{equation} we obtain an induced
$T$--action and therefore a $\Sigma|_{\sigma}$--filtration on $V$. 
We claim that as $\sigma$ varies, this induces a compatible
$\Sigma$--filtration on $V$. We will derive this from the compatibility of the $\Sigma$--filtration on the algebra $B$.
 
The $T$--action on $B$, for any fixed $\sigma$, induces a $T$--action on $(B \otimes 
K[V])^G$. Here the action of $G$ on $K[V]$ is induced by $\theta$, and the action of 
$T$ on $K[V]$ is assumed to be trivial. As $\sigma$ varies, from Lemma \ref{inter} 
it follows that these actions yield a compatible $\Sigma$--filtration on $(B \otimes 
K[V])^G$.

Since ${\rm Spec}(B)\times V$ is affine and $G$ is reductive, ${\rm Spec}(B) \times^G V \,=\, 
{\rm Spec}(( B \otimes K[V])^G) $ \cite[Theorem 1.1, page 27]{Mum}. The isomorphism \eqref{siso} induces a specific isomorphism 
$(B \otimes K[V])^G \cong K[V]$. This gives a $\Sigma$--filtration on $K[V]$. By linearity of the 
$G$--action on $V$, the cone-wise $T$--actions on $K[V]$ are determined by cone-wise $T$--actions 
on the dual $V^{\ast}$ of $V$. These determine a $\Sigma$--filtration on $V^{\ast}$, which is 
compatible as it is the restriction of the compatible $\Sigma$--filtration on $K[V]$. We have an 
induced compatible dual $\Sigma$--filtration on $V$ (see \cite{IS}, section 6.3). This 
$\Sigma$--filtration on $V$ agrees with the $\Sigma$--filtration on $V$ derived immediately after 
\eqref{siso} as the cone-wise $T$--actions match. This proves the claim regarding the compatibility of that $\Sigma$--filtration.

Note that the $T$--action on $V$, associated to the cone $\sigma$, is given by
 $\theta (\rho_{\sigma})$. Since the $\Sigma$--filtration on $V$ is compatible, 
 by Klyachko \cite{Kly}, it gives rise to a 
$T$--equivariant vector bundle over $X$ and
the ${\rm GL}(V)$--valued functions $\theta (\rho_{\sigma} \rho_{\tau}^{-1})$ extend regularly over $X_{\sigma} \bigcap X_{\tau}$. It follows that the functions
$\rho_{\sigma} \rho_{\tau}^{-1}$ extend regularly over $X_{\sigma} \bigcap X_{\tau}$. 
Since $G$ is closed in ${\rm GL}(V)$, the extensions are $G$--valued.
This allows us to construct a $T$--equivariant principal
$G$--bundle $E_G$ over $X$ by gluing the bundles $\{ E^{\sigma}_G \}$
using $\{\rho_{\sigma} \rho_{\tau}^{-1}\}$ as transition functions.

It is straightforward to show that $\mb{A}(\mb{N}_0^T(E_G)) \cong B$: As an algebra, $\mb{A}(\mb{N}_0^T(E_G)) \,=\, {\rm Spec}(B) \times^G K[G]$.
We have an isomorphism $$ \alpha_{\ast}\,:\, {\rm Spec}(B) \times^G K[G]\,\longrightarrow\, K[G] \,,$$ induced by the $G$--equivariant isomorphism $$\alpha\,:\, {\rm Spec}(B) \,\longrightarrow \,G \,, \quad {\rm where}\; \alpha(y_0) = 1_G\, .$$ Moreover, the map
$\alpha$ also induces an isomorphism $$\alpha^{\ast}: K[G] \longrightarrow B\,.$$ Thus we have an isomorphism 
$$ \alpha^{\ast} \circ \alpha_{\ast} \,:\, {\rm Spec}(B) \times^G K[G]\,\longrightarrow\, B \,.$$

Given a right $T$--action $\rho_{\sigma}$ on ${\rm Spec}(B)$, let $\rho_{\sigma}^{\ast}$ be
the induced $T$--action on $B$. These actions, as $\sigma$ varies, produce the $\Sigma$--filtration on $B$. The 
$\Sigma|_{\sigma}$--filtration or $T_{\sigma}$--action on ${\rm Spec}(B)\times^G K[G]$ is induced
from $\rho_{\sigma}$ by $\alpha$. Let us call this action $\alpha_{\ast}(\rho_{\sigma})$. Note that the action on $B$ induced from $\alpha_{\ast}(\rho_{\sigma})$ by the isomorphism $\alpha^{\ast}$ is same as $\rho_{\sigma}^{\ast}$. Thus $\alpha^{\ast} \circ \alpha_{\ast}$ induces the required isomorphism of $\Sigma$--filtrations.
\end{proof}

\begin{lemma}\label{full} Let $X$ be a toric variety such that
every maximal cone in its fan is top-dimensional. Then the functor
$$\mb{A} \,:\, \mf{Nor}^T(X) \,\longrightarrow\, \mf{Calg}_G(\Sigma)$$
(see \eqref{not3}, Definition \ref{def1} and Lemma \ref{leml}) is full. \end{lemma}

\begin{proof} 
 Let $\mb{E}^1,\, \mb{E}^2 \in \mf{Nor}^T(X)$ and $\phi: \mb{A}(\mb{E}^2) \longrightarrow
\mb{A}(\mb{E}^1)$ be a morphism in $\mf{Calg}_G(\Sigma)$. We need to 
show that there exists a morphism $\Psi \,:\, \mb{E}^1 \,\longrightarrow \,\mb{E}^2 $ in
$\mf{Nor}^T(X)$ such that $\mb{A}(\Psi)\,=\, \phi$.

Note that the underlying algebra of $\mb{A}(\mb{E}^j)$ is the algebra of functions on the
fiber of the principal $G$--bundle ${\overline{E}}^j(K[G])$ at $x_0$. Then $\phi$ induces a
$G$--equivariant morphism of varieties
$$\phi'\,:\, {\overline{E}}^1(K[G])(x_0)\,\longrightarrow\, {\overline{E}}^2(K[G])(x_0) \,.$$ The latter, by $T$--equivariance induces
a morphism of principal bundles
$$\phi'_{\ast}\,:\, {\overline{E}^1}(K[G])\,\longrightarrow\,{\overline{E}^2}(K[G])$$ over the
open orbit $O$ of $X$. 
This induces an isomorphism of the sheaves of $\mc{O}_O-$algebras
$$ \phi_{\ast}\,:\, {\overline{\mb{E}}^1}(K[G])\,\longrightarrow\,{\overline{\mb{E}}^2}(K[G]) \,.$$

For any finite dimensional $\Sigma$--filtered subspace $V$ of
$\mb{A}(\mb{E}^2)$, the restriction $$\phi\vert_V\, :\, V\,\longrightarrow \,\phi(V)$$
is an isomorphism
of $\Sigma$--filtered vector spaces. The corresponding restriction
of $\phi_{\ast}$ induces a $T$--equivariant isomorphism of vector
bundles over $O$. But since $\phi|_V$ respects the filtrations,
by the arguments of Klyachko \cite[Assertion 2.2.4]{Kly}, this extends to a $T$--equivariant
isomorphism of vector bundles over $X$. Then taking direct limit
over all finite dimensional $\Sigma$--filtered subspaces of
$\mb{A}(\mb{E}^2)$, we observe that $\phi_{\ast}$ extends over $X$ as a
$T$--equivariant isomorphism of quasi-coherent sheaves of 
$\mc{O}_X-$modules. We show that the extension is in fact an isomorphism of 
 $\mc{O}_X-$algebras.
 
For $j\,=\,1,\,2,$ let
 $$m_j\, :\, K[{\overline{\mb E}}^j(K[G])]\otimes
K[{\overline{\mb E}}^j(K[G])]\,\longrightarrow\, K[{\overline{\mb E}}^j(K[G])]$$
denote the multiplication in the sheaf ${\overline{\mb E}}^j(K[G]) $. 
Then by construction of $\phi_{\ast}$, the equality
$$ m_2 \circ (\phi_{\ast} \otimes \phi_{\ast})\,= \, \phi_{\ast} \circ m_1 \,$$
of morphisms of quasi-coherent sheaves of $\mc{O}_X-$modules, holds over $O$. Therefore, it holds over the closure $X$ of $O$.

We may similarly argue that $\phi_{\ast}$ is
$G$--equivariant. Thus $\phi_{\ast}$ induces an isomorphism $\widehat{\phi}_{\ast}\,:\, {\overline{ E}}^1(K[G])\,\longrightarrow\,
{\overline{ E}}^2(K[G])$ of $T$--equivariant principal
$G$--bundles. This induces an isomorphism
$$\mb{N}_0^T(\widehat{\phi}_{\ast})\,:\, \mb{N}_0^T ({\overline{E}}^1(K[G]))\,\longrightarrow\,
 \mb{N}_0^T ({\overline{ E}}^2(K[G]) )\,. $$

By Theorem \ref{equiv1}, there exists a functorial of isomorphism $\Phi \,:\, \mb{N}_0^T \circ \mb{N}_1^T \,\longrightarrow
\,1_{\mf{Nor}^{T}(X)}$. 
Applying it we have the required morphism
$$ \Psi \,=\, \Phi ( \mb{N}_0^T(\widehat{\phi}_{\ast}) )\,:\,\mb{\overline{E}^1}\,\longrightarrow\,
\mb{\overline{E}^2 \,.} $$ 
\end{proof}

The following theorem is a consequence of Theorem \ref{equiv1}, and Lemmas \ref{faithful}, 
\ref{surj} and \ref{full}.

\begin{theorem}\label{classi} Let $X$ be a toric variety
such that every maximal cone in its fan is top-dimensional. Then there is an equivalence of categories
between
$\mf{Pbun}_G^T(X)$ and $\mf{Calg}_G(\Sigma)$, where $G$ is a reductive group. \end{theorem}

\section{Reduction of structure group}

\begin{theorem}\label{reduc} Suppose $H$ is a reductive subgroup of $G$, and
let $E_G$ be a $T$--equivariant principal $G$--bundle over $X$.
Let $S \,=\, { \overline{\mb{E}}_{\sharp} }
(K[G])$, where $\mb{E}\,= \,\mb{N}^{T}_0(E_G)$.
Then $E_G $ admits a $T$--equivariant reduction of structure group 
to $H$ if and only if there exists a filtered algebra $R \in \mf{Calg}_H(\Sigma)$ such that 
$(R \otimes K[G])^H \cong S$ in $\mf{Calg}_G(\Sigma) $. 
\end{theorem}

\begin{proof} If $E_G$ admits a $T$--equivariant reduction of structure group to $E_H \in \mf{Pbun}^T_H(X) $, there exists 
an isomorphism 
\begin{equation}\label{iso6}
E_H \times_H G \,\cong\, E_G \,.
\end{equation} 
Let $R $ be the $\Sigma$--filtered algebra in $\mf{Calg}_H(\Sigma)$ associated to
$E_H$. Then \eqref{iso6} yields 
an isomorphism 
$$(R \otimes K[G] )^H \,\cong\, S
$$
in $\mf{Calg}_G(\Sigma)$.

On the other hand, a bundle $E_G$ with filtered algebra $(R \otimes K[G] )^H$ is isomorphic to $E_H \times_H G$ where $E_H$ is the
bundle associated to $R$. This follows from the isomorphism at the level of fibers at the closed point $x_0$ in the open $T$--orbit $O$,
together with the isomorphism of filtrations, as in the proof of Lemma \ref{full}.
\end{proof}

\section*{Acknowledgements}

We are very grateful to the two referees for their valuable comments and suggestions that led to a significant 
improvement of the manuscript. It is a pleasure to thank Paul Bressler, Pralay Chatterjee, and Nathan Ilten 
for useful discussions. The first-named author is partially supported by a J. C. Bose Fellowship. The 
second-named author is partially supported by a research grant from NBHM and SERB. The last-named author is partially 
supported by an SRP grant from METU NCC.

\end{document}